\documentclass[12pt]{amsart}
\usepackage{amssymb}
\usepackage[latin1]{inputenc}
\usepackage{amsthm}
\usepackage{amsfonts}
\usepackage{mathrsfs}
\usepackage{graphicx}
\usepackage{amsmath}
\usepackage[all]{xy}

\let\mathcal\mathscr
\def\bB{{\mathbb B}}
\def\bC{{\mathbb C}}

\def\bH{{\mathbb H}}

\def\bR{{\mathbb R}}

\def\bQ{{\mathbb Q}}
\def\bZ{{\mathbb Z}}
\newtheorem{thm}{Theorem}[section]

\def\tilde{\widetilde}

\def\phi{\varphi}

\numberwithin{equation}{section}

\newtheorem{conjecture}[thm]{Conjecture}

\newtheorem{proposition}[thm]{Proposition}

\begin{document}

\title {Hyperbolicity, Automorphic forms and Siegel modular varieties}

\author{Erwan Rousseau}
\thanks{Partially supported by the ANR project \lq\lq POSITIVE\rq\rq{}, ANR-2010-BLAN-0119-01.
This work has been carried out in the framework of the Labex Archim\`ede (ANR-11-LABX-0033) and of the A*MIDEX project (ANR-11-IDEX-0001-02), funded by the ``Investissements d'Avenir" French Government programme managed by the French National Research Agency (ANR)}
\keywords{entire curves, Kobayashi hyperbolicity, bounded symmetric domains, automorphic forms.}
\subjclass[2010]{Primary: 32Q45, 32M15; Secondary: 11G99, 14K15.}
\date{}

\begin{abstract}
We study the hyperbolicity of compactifications of quotients of bounded symmetric domains by arithmetic groups. We prove that, up to a finite \'etale cover, they are Kobayashi hyperbolic modulo the boundary. Applying our techniques to Siegel modular varieties, we improve some former results of Nadel on the non-existence of certain level structures on abelian varieties over complex function fields.

\end{abstract}

\maketitle

\section{Introduction}
Any complex space $Y$ can be equipped with an intrinsic pseudo-distance $d_Y$, the Kobayashi pseudo-distance \cite{kob}. It is the largest pseudo-distance such that every holomorphic map $f: \Delta \to Y$ from the unit disc, equipped with the Poincar\'e metric, is distance decreasing.

Recall \cite{kob} that a complex space $Y$ is said to be (Kobayashi) hyperbolic modulo $W \subset Y$ if for every pair of distinct points $p,q$ of $Y$ we have $d_Y(p,q)>0$ unless both are contained in $W$. 
In the case where $Y$ is compact and $W=\emptyset$, Brody's lemma gives a criterion for hyperbolicity: $Y$ is hyperbolic if and only if there are no curves $f: \bC \to Y$.

In the case where $W \neq \emptyset$, we do not know an analogue of Brody's criterion for hyperbolicity modulo $W$. It is not known whether $Y$ is necessarily hyperbolic modulo $W$ if it is Brody hyperbolic modulo $W$ i.e. every entire curve $f: \bC \to Y$ has image in $W$. The reason is that the classical Brody's lemma does not provide any information on the image of the entire curve.

In this article, we investigate hyperbolic properties of compactifications of quotients of bounded symmetric domains by arithmetic groups. It is well known that such quotients may be far from being hyperbolic, as rational Hilbert modular surfaces show.

Let $X:=\bB/ \Gamma$ be a quotient of a bounded symmetric domain by an arithmetic subgroup $\Gamma \subset \operatorname{\operatorname{Aut}}(\bB)$. By passing to a subgroup of finite index if necessary, we will always suppose that $\Gamma$ is neat \cite{B}. In particular, $X$ admits a smooth toroidal compactification $\overline{X}$ where $D=\overline{X} \setminus X$ is a normal crossings divisor \cite{AMRT}. Moreover, $\overline{X}$ may be chosen to be projective \cite{AMRT}.

We obtain the following result
\begin{thm}\label{hypquot}
Let $X:=\bB/ \Gamma$ be a quotient of a bounded symmetric domain by a neat arithmetic subgroup $\Gamma \subset \operatorname{Aut}(\bB)$. Then there exists a finite \'etale cover $X_1=\bB/ \Gamma_1$ determined by a subgroup $\Gamma_1 \subset \Gamma$ such that $\overline{X_1}$, a smooth projective compactification of $X_1$, is Kobayashi hyperbolic modulo $D_1:=\overline{X_1} \setminus X_1$.
\end{thm}

As a corollary, we obtain a new simple proof of the main result of \cite{Nad}, where it is proved that $\overline{X_1}$ is Brody hyperbolic modulo $D_1$: the image $f(\bC)$ of any non-constant holomorphic map $f: \bC \to \overline{X_1}$ is contained in $D_1$.

While Nadel's proof \cite{Nad} was based on Nevanlinna theory, our proof consists in using the fundamental distance decreasing property of the Kobayashi pseudo-metric mentioned above. To prove Theorem \ref{hypquot}, we will construct pseudo-distances on $\overline{X_1}$ for which holomorphic maps from the unit disc are distance decreasing.

As an application, we study Siegel modular varieties from this point of view. Their geometry has attracted a lot of attention (see \cite{smv} for a survey). In particular, the birational geometry of $\mathcal{A}_g$ and $\mathcal{A}_g(n)$, the moduli spaces of principally polarized abelian varieties with level structures, has been extensively investigated. The Kodaira dimension of compactifications $\overline{\mathcal{A}_g(n)}$ has been studied by Tai, Freitag, Mumford, Hulek and others, proving the following result.

\begin{thm}[Tai, Freitag, Mumford, Hulek]
$\overline{\mathcal{A}_g(n)}$ is of general type for the following values of $g$ and $n\geq n_0$:

$$\begin{tabular}{l|ccccccc}
  \hline
  g &  $1$ & $2$ & $3$ & $4$ & $5$ & $6$ & $\geq 7$ \\
  \hline
  $n_0$ & $7$ & $4$ & $3$ & $2$ & $2$ & $2$ & $1$  \\
  \hline
\end{tabular}$$
\end{thm}

Therefore it is very natural to study hyperbolicity of these spaces in the light of the Green-Griffiths-Lang conjecture
\begin{conjecture}[Green-Griffiths, Lang]
Let $X$ be a projective variety of general type. Then there exists a proper algebraic subvariety $Y \subset X$ such that $X$ is hyperbolic modulo $Y$.
\end{conjecture}

In this setting, using the theory of automorphic forms, we can apply the same strategy as above and obtain

\begin{thm}\label{siegel}
$\overline{\mathcal{A}_g(n)}$ is hyperbolic modulo $D:=\overline{\mathcal{A}_g(n)}\setminus \mathcal{A}_g(n)$ if
$n >6g.$
\end{thm}

As a corollary we have
\begin{thm}\label{siegel3}
If $A$ is a principally polarized abelian variety of dimension $g \geq 1$  which is defined and non-constant over $k$, a complex function field of genus $\leq 1$, then $A$ does not admit a level-$n$ structure for $n >6g.$
\end{thm}

In particular, this improves the previous bound, $n\geq \max (\frac{1}{2}g(g+1), 28)$, obtained by Nadel in \cite{Nad}.

\subsubsection*{Acknowledgments}
The author is grateful to Carlo Gasbarri and Xavier Roulleau for several interesting and fruitful discussions and remarks. He also thanks CNRS for the opportunity to spend a semester in Montreal at the UMI CNRS-CRM and the hospitality of UQAM-Cirget where part of this work was done.

\section{Distance decreasing pseudo-distances}
Let $X:=\bB/ \Gamma$ be a quotient of a bounded symmetric domain by an arithmetic subgroup $\Gamma \subset \operatorname{Aut}(\bB)$. We first give a criterion to ensure the existence of a pseudo-distance $\tilde{g}$ on $\overline{X}$ non-degenerate at a given $x \in X$ and satisfying the distance decreasing property: every holomorphic map $f: (\Delta, g_P) \to (\overline{X}, \tilde{g})$ from the unit disc, equipped with the Poincar\'e metric $g_P$, is distance decreasing i.e.
$$f^*\tilde{g} \leq g_P.$$

Let $g$ denote the Bergman metric on $\bB$ satisfying the K\"ahler-Einstein property $\operatorname{Ric}(g)=-g$. Let $\gamma >0$ be a rational number such that $g$ has holomorphic sectional curvature $\leq -\gamma$. We want to construct a continuous function $\psi: \overline{X} \to \bR^+$ such that $\psi.g$ defines a pseudo-metric on $\overline{X}$.

\begin{thm}\label{pseudo}
Let $x \in X$ and suppose there is a section $s \in H^0(\overline{X}, l(K_{\overline{X}}+D))$ such that
\begin{itemize}
\item $s(x) \neq 0$
\item $s$ vanishes on $D$ with multiplicity $m > \frac{l}{\gamma}.$
\end{itemize}
Then there exists a pseudo-metric $\tilde{g}$ on $\overline{X}$ non-degenerate at $x$ satisfying the distance decreasing property. In particular, the Kobayashi pseudo-metric on $\overline{X}$ is non-degenerate at $x$.
\end{thm}

We will prove this theorem in $3$ steps. 

From \cite{mum}, we know that the Bergman metric induces a \emph{good} singular metric $h:=(\det g)^{-1}$ on $K_{\overline{X}}+D$. Let $x \in X$ and $s$ be as given above. Let $0< \epsilon$ be such that
$$\gamma-\epsilon > \frac{l}{m}.$$
We define
$$\psi:= ||s||_h^{\frac{2(\gamma-\epsilon)}{l}}.$$

Despite the singularities of $h$, we have
\begin{proposition}
$\psi$ is a continuous function on $\overline{X}$ vanishing on $D$.
\end{proposition}

\begin{proof}
We know \cite{mum} that $K_{\overline{X}}+D$ is a \emph{good extension} of $K_{\overline{X}}$ with the metric $h$. This implies that near $D$, which locally is given by $\{z_1 \dots z_k =0\}$, we have
\begin{equation}\label{good}
||s||_h=O(|z_1\dots z_k|^m \log^{2N} |z_1\dots z_k|),
\end{equation}
which suffices to obtain the conclusion. 
\end{proof}

The second step is
\begin{proposition}
$\psi \cdot g$ is a continuous pseudo-metric on $\overline{X}$ vanishing on $D$.
\end{proposition}
\begin{proof}
We know that $g$ has Poincar\'e growth near $D$ (Lemma 2.1 \cite{Nad}) i.e.
$$g=O(g_P),$$
where $$g_P:=\sum_{i=1}^{k} \frac{dz_i\otimes d\overline{z}_i}{|z_i|^2\log^2|z_i|^2}+\sum_{i=k+1}^{n} dz_i\otimes d\overline{z}_i$$ is the metric obtained by taking the Poincar\'e metric on the punctured disc.
Therefore if we combine this with (\ref{good}), we have
$$||s||_h^{\frac{2(\gamma-\epsilon)}{l}} \cdot g_{ij}=O(|z_1\dots z_k|^{\frac{2m(\gamma-\epsilon)}{l}-2} \log^{\alpha} |z_1\dots z_k|).$$
This estimate clearly implies the proposition.
\end{proof}

Finally as a third step, we are in position to construct a distance decreasing pseudo-distance which finishes the proof of Theorem \ref{pseudo}.
\begin{proposition}
There exists a constant $\beta>0$ such that $\tilde{g}:=\beta \cdot \psi \cdot g$ is a distance decreasing pseudo-distance: for any holomorphic map $f: \Delta \to \overline{X}$ from the unit disc equipped with the Poincar\'e metric $g_P$, we have
$$f^*\tilde{g}\leq g_P.$$
\end{proposition}

\begin{proof}
Let $$u(z):=(1-|z|^2)^2 \cdot f^*\tilde{g}.$$ Restricting to a smaller disc, we may assume that $u$ vanishes on the boundary of $\Delta$. Take a point $z_0$ at which $u$ is maximum. $z_0 \in \Delta \setminus f^{-1}(\{\psi=0\})$ otherwise $u$ would vanish identically. Therefore at $z_0$ we have
$$i\partial \overline{\partial} \log u(z)_{z=z_0} \leq 0.$$
We have
$$i\partial \overline{\partial} \log u(z)=i\partial \overline{\partial} \log f^*\psi+i\partial \overline{\partial} \log f^*g+i\partial \overline{\partial} \log (1-|z|^2)^2.$$
On $X$, we have
$$i\partial \overline{\partial} \log \psi=(\gamma-\epsilon) \operatorname{Ric}(g)=-(\gamma-\epsilon)g$$
by the K\"ahler-Einstein property.

$g$ has holomorphic sectional curvature $\leq -\gamma$ and therefore
$$i\partial \overline{\partial} \log f^*g \geq \gamma f^*g.$$
Moreover
$$i\partial \overline{\partial} \log (1-|z|^2)^2=\frac{-2}{(1-|z|^2)^2}.$$
 Combining the above inequalities, we obtain
$$ -(\gamma-\epsilon) f^*g(z_0)+\gamma f^*g-\frac{2}{(1-|z_0|^2)^2}\leq 0.$$
So
$$\epsilon/2.f^*g(z_0) \cdot (1-|z_0|^2)^2 \leq 1.$$

Finally, we have
$$u(z)\leq u(z_0)=(1-|z_0|^2)^2 \cdot \beta\psi(f(z_0)) \cdot f^*g(z_0)\leq \frac{2\beta\psi(f(z_0))}{\epsilon}.$$
Therefore if we take
$$\beta=\frac{\epsilon/2}{\sup \psi},$$
we obtain the distance decreasing property.
\end{proof}

\section{Proof of Theorem \ref{hypquot}}
To prove Theorem \ref{hypquot} it is now sufficient to show the existence of an \'etale cover $X_1=\bB/ \Gamma_1$
satisfying the hypothesis of Theorem \ref{pseudo} for every $x \in X_1$.

First we recall the following result
\begin{proposition}[\cite{HTO}, 4.2]\label{hwang}
For any point $x \in X$, there exists a section $$s \in H^0(\overline{X}, (\binom {n+1} {2} +2) (K_{\overline{X}}+D))$$ such that $s(x) \neq 0$ and $s|_D \equiv 0$, where $D:=\overline{X} \setminus X$ and $n=\dim X$.
\end{proposition}

Let $l:=\binom {n+1} {2} +2$ and take a positive integer $m$ such that
$m> \frac{l}{\gamma}.$

Now, we have

\begin{proposition}\label{nv}
There exists a covering  $X_1$ of $X$ such that for any $x \in X_1$, there is a section $s \in H^0(\overline{X}_1, l(K_{\overline{X_1}}+D_1))$ vanishing on $D_1:=\overline{X_1} \setminus X_1$ with multiplicity at least $m$ such that $s(x) \neq 0$.
\end{proposition}

\begin{proof}
By \cite{mum}, one can find a covering $X_1$ such that $\pi: \overline{X_1} \to \overline{X}$ is ramified to order at least $m$ along $D_1$. Let $x \in X_1$.
From Proposition \ref{hwang}, there is a section $s_0 \in H^0(\overline{X}, l (K_{\overline{X}}+D))$ such that $s_0(\pi(x)) \neq 0$ and $s_0|_D \equiv 0$. We just have to take
$$s:=\pi^*(s_0) \in H^0(\overline{X_1}, l(K_{\overline{X_1}}+D_1)-mD_1).$$
\end{proof}

As a consequence, we obtain a proof of Theorem \ref{hypquot}.

\begin{proof}
Proposition \ref{nv} and Theorem \ref{pseudo} imply that for any $x \in X_1$, there is a pseudo-distance $\rho$ non-zero at $x$ such that 
$$d_{\overline{X_1}}\geq \rho.$$
This implies that $\overline{X_1}$ is Kobayashi hyperbolic modulo $D_1$.
\end{proof}

\section{Siegel modular varieties}
We will apply Theorem \ref{pseudo} to the particular case of quotients of the Siegel upper half space by principal congruence subgroups.

Here $\bB=\bH_g$, the Siegel upper half space of rank $g$. The arithmetic group $\Gamma$ is $\operatorname{Sp}(2g, \bZ)$ and $
\Gamma(n)$ is the kernel of $\operatorname{Sp}(2g, \bZ) \to \operatorname{Sp}(2g, \bZ/n)$.

The space $\mathcal{A}_g(n):=\bH_g / \Gamma(n)$ is the moduli space of principally polarized abelian varieties with a level-$n$ structure. Recall that if $A$ is an abelian variety of dimension $g$ over a field $k$ of characteristic $0$, a level-$n$ structure is a $2g$-tuple $(x_1,\dots,x_{2g})$ of points in $A(k)$ which generate the subgroup of all $n$-torsion points in $A(\overline{k})$ and form a symplectic basis with respect to the Weil pairing.

It is known that the Bergman metric on $\bH_g$ has holomorphic sectional curvature $\leq \frac{-2}{g(g+1)}$. We can therefore take $$\gamma= \frac{2}{g(g+1)}.$$

We are going to see that in this situation, we can use the theory of automorphic forms to obtain the section needed to apply Theorem \ref{pseudo}.

First, recall (see for example \cite{smv}) that the automorphy factor of modular forms on $\bH_g$ defines a line bundle $L$ on the quotient $X:=\bH_g / \Gamma(n)$, the ($\bQ$-) line bundle of modular forms of weight $1$. If the spaces $\overline{X}=\overline{\mathcal{A}_g(n)}$ are smooth (for example $\Gamma(n)$ is neat when $n \geq 3$), we have
$$K_{\overline{X}}=(g+1)L-D,$$
where $D$ is the boundary divisor.

Let $F$ be a modular form with respect to the full modular group $\operatorname{Sp}(2g, \bZ)$. Then following \cite{Hulek}, we define the order $o(F)$ of $F$ as the quotient of the vanishing order of $F$ at the boundary divided by the weight of $F$.

Particularly useful for us is the following result of Weissauer on the existence of cusp forms of high order which do not vanish at a given point of the Siegel space.

\begin{thm}[\cite{We}]
For every point $\tau \in \bH_g$ and every $\epsilon>0$ there exists a modular form $F$ of order $o(F) \geq \frac{1}{12 + \epsilon}$ which does not vanish at $\tau.$
\end{thm}

We obtain as a corollary
\begin{proposition}\label{siegel2}
Let $X:=\mathcal{A}_g(n)$, $x\in X$ and suppose $n >6g$. Then there is a section $s \in H^0(\overline{X}, l(K_{\overline{X}}+D))$ such that
\begin{itemize}
\item $s(x) \neq 0$
\item $s$ vanishes on $D$ with multiplicity $m > \frac{l}{\gamma}.$
\end{itemize}
\end{proposition}

\begin{proof}
We consider the natural map
$$\pi: \overline{\mathcal{A}_g(n)} \to \overline{\mathcal{A}_g}.$$ This map is ramified to order at least $n$ over the divisor at infinity \cite{Nad}. We choose some $\epsilon>0$ such that $\epsilon < \frac{2}{g}$. Now, we take a modular form given by Weissauer's theorem non-vanishing at $\pi(x)$. In terms of line bundles, it gives a section $\widetilde{s} \in H^0(\overline{\mathcal{A}_g}, kL-mD)$ such that $\frac{m}{k}\geq \frac{1}{12 + \epsilon}$ non-vanishing at $\pi(x)$. We take $s:=\pi^*(\widetilde{s}).$ It gives a section in $H^0(\overline{X}, kL-mnD)$ non-vanishing at $x$, where we make the slight abuse of denoting the line bundles corresponding to the modular forms and the boundary again by $L$ and $D$. Therefore, using the equality $K_{\overline{X}}=(g+1)L-D,$ we have obtained a section in $H^0(\overline{X}, l(K_{\overline{X}}+D))$ with $l=\frac{k}{g+1}$ vanishing on $D$ with multiplicity $mn$. Finally, we observe that
$$mn\geq \frac{(6g+1)k}{12+\epsilon}> \frac{kg}{2}=\frac{l}{\gamma}.$$

\end{proof}

Theorem \ref{siegel} is an immediate corollary of Proposition \ref{siegel2} and Theorem \ref{pseudo}.

\

Theorem \ref{siegel3} is a consequence of Theorem \ref{siegel}, since a non-constant principally polarized abelian variety over $k$, a complex function field of genus $\leq 1$, which admits a level-$n$ structure induces a non-constant holomorphic map $\bC \to \overline{\mathcal{A}_g(n)}$ whose image does not lie in $D$.

\bigskip

\noindent
Erwan Rousseau \\
erwan.rousseau@univ-amu.fr\\ 
Aix Marseille Universit\'e,\\
CNRS, Centrale Marseille, I2M, UMR 7373,\\
13453 Marseille, France

\end{document}